\documentclass[12pt]{amsart}
\usepackage[bottom=3cm, left=3cm, right=3cm, includefoot]{geometry}

\usepackage{amsthm}
\newtheorem{theorem}{Theorem}[section]
\newtheorem{lemma}[theorem]{Lemma}
\newtheorem{corollary}[theorem]{Corollary}
\newtheorem{definition}[theorem]{Definition}
\newtheorem{remark}[theorem]{Remark}
\newtheorem{example}[theorem]{Example}

\usepackage[pdftex]{graphicx}
\usepackage{amscd}
\usepackage{amssymb}
\usepackage{comment}
\usepackage{enumerate}


\begin{document}
\title{Minimal genus problem for $T^{2}$-bundles over surfaces}
\author{Reito Nakashima}
\date{\today}
\maketitle

\begin{abstract}
For any positive integer $g$, we completely determine the minimal genus function for $\Sigma_{g}\times T^{2}$. We show that the lower bound given by the adjunction inequality is not sharp for some class in $H_{2}(\Sigma_{g}\times T^{2})$. However, we construct a suitable embedded surface for each class and we have exact values of minimal genus functions.
\end{abstract}

%------------------------------------------------

\section{introduction}

\subsection{Minimal genus functions}
Let $M$ be a smooth closed oriented $4$-manifold. It is well-known that any homology class $\sigma$ in $H_{2}(M)$ is represented by a connected oriented smoothly embedded surface $\Sigma\subset M$. For each class $\sigma$ in $H_{2}(M)$, we want to determine the minimal genus of a surface which represents the class $\sigma$.
\begin{definition}
Let $M$ be a smooth closed oriented $4$-manifold. The minimal genus function $G:H_{2}(M)\to\mathbb{Z}$ is defined for each class $\sigma\in H_{2}(M)$ by\[ G(\sigma):=\min\{g(\Sigma)\mid\Sigma\subset M \mbox{ a connected smooth embedded surface representing } \sigma\}.\]
\end{definition}
In general, calculating minimal genus functions is a difficult problem and there are not so many examples of $4$-manifolds whose minimal genus functions are completely known.

In 1994, P. B. Kronheimer and T. S. Mrowka~\cite{km} solved this problem for the complex projective plane, known as the Thom conjecture, using Seiberg-Witten theory. Their result says that the minimal genus of a surface representing $dh\in H_{2}(\mathbb{CP}^{2})\cong\mathbb{Z}$ is $\frac{1}{2}(|d|-1)(|d|-2)$, where $d$ is a non-zero integer and $h\in H_{2}(\mathbb{CP}^{2})$ is a generator. 

Other examples are given by Bang-He Li and Tian-Jun Li~\cite{bt2}. They determined the minimal genus functions for $S^{2}$-bundles over closed oriented surfaces completely.

See Terry Lawson's survey~\cite{t} for more about minimal genus problems.

%------------------------------------------------
\newpage

\subsection{Main results}
Our main theorem is

\begin{theorem}\label{theorem:1}
Let $M$ be $\Sigma_g\times T^{2}$ and let $\sigma$ be a class in $H_{2}(M)$, where $\Sigma_{g}$ is an oriented closed surface of genus $g\geq 1$. We have
\begin{eqnarray*}
G(\sigma)=\left\{\begin{array}{ll}
0 & (\sigma=0) \\
1+\dfrac{1}{2}|\sigma\cdot\sigma|+(g-1)|\sigma\cdot F| & (*) \\
2 & (otherwise) \\ 
\end{array}\right.,
\end{eqnarray*}
where $F=[\{*\}\times T^2]\in H_{2}(M)$ and the condition $(*)$ means that one of the following conditions is satisfied.
\begin{itemize}
\item$F\cdot\sigma\neq 0$.
\item$\sigma\cdot\sigma\neq 0$.
\item$\sigma\neq 0 \mbox{ and } \sigma=u\otimes v+n(-F)\mbox{ for } \mbox{ some } u\in H_{1}(\Sigma_{g}), v\in H_{1}(T^{2}) \mbox{ and } n\in\mathbb{Z}$.
\end{itemize}
\end{theorem}
Using the adjunction inequality, we have the lower bound \[G(\sigma)\geq1+\dfrac{1}{2}|\sigma\cdot\sigma|+(g-1)|\sigma\cdot F|\] for every class $\sigma$ in $H_{2}(M)\setminus\{0\}$. If a class $\sigma$ is in the last exceptional case, we have $G(\sigma)\geq 1$, whereas we show that there are no embedded surface that gives the equality $G(\sigma) = 1$. In this case we improve the lower bound by $1$. That is, the adjunction inequality is not sharp. See Lemma \ref{lemma:3} for details.

Since the minimal genus function on $H_{2}(M)$ is invariant under actions induced by self-diffeomorphisms of $M$, it suffices to show, for each orbit of these actions, the above equality for a representative of the orbit. For each representative, we construct a connected embedded surface representing the class whose genus is as in the theorem by the circle sum operation. The circle sum operation is an operation which constructs a new connected embedded surface from two connected embedded surfaces in a $4$-manifold. We discuss about circle sum operation in Section \ref{section:1}. See also Bang-He Li and Tian-Jun Li~\cite{bt1} for details.

In Section \ref{section:2}, we explain some corollaries related to our main result. The first corollary is about complexity of embedded surfaces. We interpret our main result to complexity of connected surfaces, and then we compare with the minimal complexity functions which allow disconnected closed surfaces for representing surfaces. In our case, non-sharpness of the adjunction inequality gives difference between the connected version of the minimal complexity functions and the disconnected version.
 
The second is the result for some non-trivial $T^{2}$-bundles over surfaces. Let $N$ be a non-trivial $S^{1}$-bundle over a genus $g$ surface and let $M=N\times S^{1}$. In this case, our constructions of surfaces embedded in $\Sigma_{g}\times T^{2}$ also work and we get exact values of the minimal genus function completely.

Finally, we observe automorphisms on $H_{2}(M)$ for $M=\Sigma_{g}\times T^{2}$ with $g\geq 2$. Let $\mathcal{H}$ be the subgroup of the automorphism group ${\rm Aut}(H_{2}(M))$ defined by $\mathcal{H}=\{\phi\in{\rm Aut}(H_{2}(M))\mid\phi^{*}Q=Q \mbox{ and }\phi^{*}G=G\}$, where $Q$ is the intersection form of $M$, and let $\theta:{\rm Diff}^{+}(M)\to\mathcal{H}$ be the obvious homomorphism from the group of orientation preserving diffeomorphisms of $M$. We show that $\mathcal{H}/{\rm Im}\theta\cong\mathbb{Z}/2\mathbb{Z}$ and give an explicit generator set for ${\rm Im}\theta$.

In Section \ref{section:3} we show that there are topologically locally-flat surfaces in $\Sigma_{g}\times T^{2}$ whose genera are strictly smaller than the result stated in Theorem \ref{theorem:1}.

%------------------------------------------------

\subsection{Notations}
In this paper, we always assume that $4$-manifolds and surfaces are oriented and assume that $g$ is a positive integer. All homology groups have coefficients in $\mathbb{Z}$. 

%------------------------------------------------

\section{preliminaries}
\subsection{Self-diffeomorphisms of $M$ and induced actions on $H_{2}(M)$}\label{section:0}

First, we consider the second homology group of $M$. By the K\"{u}nneth formula, we have \[H_{2}(M)\cong(H_{1}(\Sigma_{g})\otimes H_{1}(T^{2}))\oplus H_{2}(\Sigma_{g})\oplus H_{2}(T^{2})\cong\mathbb{Z}^{4g+2}.\]

We fix an orientation of $\Sigma_{g}$ and take a symplectic basis $x_{1}, z_{1}, x_{2}, z_{2}, \dots, x_{g}, z_{g}$ of $H_{2}(\Sigma_{g})$ so that the basis satisfies $x_{i}\cdot z_{j} = \delta_{ij}$, $x_{i}\cdot x_{j} = 0$ and $z_{i}\cdot z_{j} = 0$ for any $1\leq i, j\leq g$, where $\delta_{ij}$ is the Kronecker delta. We identify each element with a loop embedded in $\Sigma_{g}$ as in the Figure \ref{figure:1}. Similarly, we also fix an orientation of $T^{2}$, take a symplectic basis $y, t$ of $H_{2}(T^{2})$ with $y\cdot t = 1$ and identify the element $y$ with a loop $S^{1}\times \{*\}$ and the element $t$ with a loop $\{*\}\times S^{1}$.

%fig1
\begin{figure}[h]
\centering
\includegraphics[width=12cm]{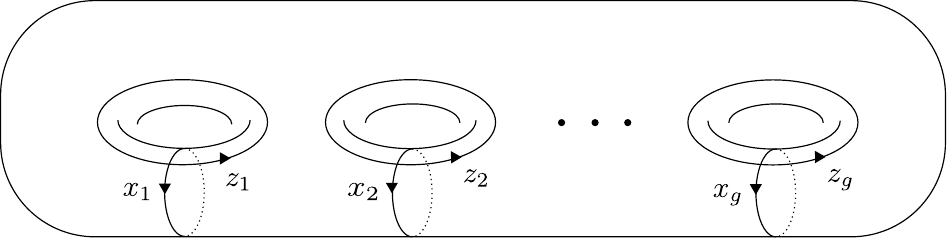}
\caption{\label{figure:1}}
\end{figure}

For $u \in H_{1}(\Sigma_{g})$ and $v \in H_{1}(T^{2})$, we denote $u\otimes v \in H_{2}(M)$ by $T_{uv}$. Furthermore, we denote $[\Sigma_{g}\times\{*\}]\in H_{2}(M)$ by $S$ and denote $[\{*\}\times T^2]\in H_{2}(M)$ by $F$. Then we take  \[T_{x_{1}y}, T_{z_{1}t}, T_{x_{1}t}, (-T_{z_{1}y}), \dots ,T_{x_{g}y}, T_{z_{g}t}, T_{x_{g}t}, (-T_{z_{g}y}), S, (-F)\] as a basis of $H_{2}(M)$. We fix the orientation of $M$ by ${T_{x_{1}y}\cdot T_{z_{1}t}=1}$. The intersection form $Q$ of $M$ is $H^{\oplus 2g+1}$, where $H$ is $\left(\begin{array}{cc}0 & 1 \\1 & 0 \end{array}\right)$.

We identify 
\[\sum_{i=1}^{g}(a_{i}T_{x_{i}y}+b_{i}T_{z_{i}t}+c_{i}T_{x_{i}t}+d_{i}(-T_{z_{i}y}))+eS+f(-F)\in H_{2}(M)\]
 with $(a_{1}, b_{1}, c_{1}, d_{1}, \dots, a_{g}, b_{g}, c_{g}, d_{g}, e, f)\in \mathbb{Z}^{4g+2}$.

Next, we construct self-diffeomorphisms of $M$ and consider these actions on the second homology group of $M$.

Let $\gamma_\alpha$ be a simple closed curve representing a primitive class $\alpha\in H_{1}(\Sigma_{g})$. Then, for each primitive class $\alpha$ in $H_{1}(\Sigma_{g})$, we have a diffeomorphism \[R_\alpha:\Sigma_{g}\to\Sigma_{g}\] defined by a right handed Dehn twist along $\gamma_\alpha$. For any class $\tau\in H_{1}(\Sigma_{g})$, we have \[(R_{\alpha})_{*}(\tau)=\tau+(\alpha\cdot\tau)\alpha\] and hence, we have
\begin{eqnarray*}
(R_{z_{i}})_{*}(x_{i}) & = & x_{i}-z_{i}, \\ 
(R_{z_{i}})_{*}(z_{i}) & = & z_{i}, \\
(R_{x_{i}})_{*}(x_{i}) & = & x_{i}, \\ 
(R_{x_{i}})_{*}(z_{i}) & = & z_{i}+x_{i}, \\
(R_{z_{i}+z_{j}})_{*}(x_{i}) & = & x_{i}-z_{i}-z_{j}, \\ 
(R_{z_{i}+z_{j}})_{*}(x_{j}) & = & x_{j}-z_{i}-z_{j}, \\
(R_{z_{i}+x_{j}})_{*}(x_{i}) & = & x_{i}-z_{i}-x_{j}, \\ 
(R_{z_{i}+x_{j}})_{*}(z_{j}) & = & z_{j}+z_{i}+x_{j},
\end{eqnarray*}
where $i\neq j$. 

Now, we have a self-diffeomorphism of $M=\Sigma_g\times T^{2}$ by $R_\alpha\times Id_{T^{2}}$. For the rest of this paper, we denote $R_\alpha\times Id_{T^{2}}$ by $R_\alpha$ for simplicity. The map $R_\alpha$ induces an isomorphism $(R_{\alpha})_{*}:H_{2}(M)\to H_{2}(M)$.

We have
\begin{eqnarray*}
(R_{z_{i}})_{*}(T_{x_{i}y}) & = & T_{x_{i}y}+(-T_{z_{i}y}), \\ 
(R_{z_{i}})_{*}(T_{x_{i}t}) & = & T_{x_{i}t}-T_{z_{i}t}, \\
(R_{x_{i}})_{*}(T_{z_{i}t}) & = & T_{z_{i}t}+T_{x_{i}t}, \\ 
(R_{x_{i}})_{*}(-T_{z_{i}y}) & = & (-T_{z_{i}y})-T_{x_{i}y}, \\
(R_{z_{i}+z_{j}})_{*}(T_{x_{i}y}) & = & T_{x_{i}y}+(-T_{z_{i}y})+(-T_{z_{j}y}), \\ 
(R_{z_{i}+z_{j}})_{*}(T_{x_{i}t}) & = & T_{x_{i}t}-T_{z_{i}t}-T_{z_{j}t}, \\
(R_{z_{i}+z_{j}})_{*}(T_{x_{j}y}) & = & T_{x_{j}y}+(-T_{z_{i}y})+(-T_{z_{j}y}), \\ 
(R_{z_{i}+z_{j}})_{*}(T_{x_{j}t}) & = & T_{x_{j}t}-T_{z_{i}t}-T_{z_{j}t}, \\
(R_{z_{i}+x_{j}})_{*}(T_{x_{i}y}) & = & T_{x_{i}y}+(-T_{z_{i}y})-T_{x_{j}y}, \\ 
(R_{z_{i}+x_{j}})_{*}(T_{x_{i}t}) & = & T_{x_{i}t}-T_{z_{i}t}-T_{x_{j}t}, \\
(R_{z_{i}+x_{j}})_{*}(T_{z_{j}t}) & = & T_{z_{j}t}+T_{z_{i}t}+T_{x_{j}t}, \\ 
(R_{z_{i}+x_{j}})_{*}(-T_{z_{j}y}) & = & (-T_{z_{j}y})+(-T_{z_{i}y})-T_{x_{j}y}.
\end{eqnarray*}

Let $\sigma$ be $(a_{1}, b_{1}, c_{1}, d_{1}, \dots, a_{g}, b_{g}, c_{g}, d_{g}, e, f)\in \mathbb{Z}^{4g+2}$. By the above computation, we have
\begin{eqnarray*}
(R_{z_{i}})_{*}(\sigma) & = & (\dots, a_{i}, b_{i}-c_{i}, c_{i}, d_{i}+a_{i}, \dots), \\
(R_{x_{i}})_{*}(\sigma) & = & (\dots, a_{i}- d_{i}, b_{i}, c_{i}+ b_{i}, d_{i}, \dots), \\
(R_{z_{i}+z_{j}})_{*}(\sigma) & = & (\dots, a_{i}, b_{i}-c_{i}-c_{j}, c_{i}, d_{i}+a_{i}+a_{j}, \\
&& \ \dots, a_{j}, b_{j}-c_{i}-c_{j}, c_{j}, d_{j}+a_{i}+a_{j}, \dots), \\
(R_{z_{i}+x_{j}})_{*}(\sigma) & = & (\dots, a_{i}, b_{i}-c_{i}+b_{j}, c_{i}, d_{i}+a_{i}+d_{j}, \\
&& \ \dots, a_{j}-a_{i}-d_{j}, b_{j}, c_{j}-c_{i}+b_{j}, d_{j}, \dots).
\end{eqnarray*}

We need more diffeomorphisms to simplify classes in $H_{2}(M)$ sufficiently.

We define a diffeomorphism $D_{x_{i}y}:M\to M$ as follows. Let $U\cong I\times \mathbb{R}/\mathbb{Z}$ be a closed tubular neighborhood of the loop $z_{i}\subset\Sigma_{g}$, where $I$ is the closed interval $[0,1]$. We define a self-diffeomorphism $D_{x_{i}y}$ of $U\times T^{2}\cong I\times\mathbb{R}/\mathbb{Z}\times(\mathbb{R}/\mathbb{Z})^{2}\subset M$ by 
\[D_{x_{i}y}(x, z, y, t)=\left(x, z, y+\frac{\lambda(x)}{\lambda(1)}, t\right),\]
where $\lambda:I\to \mathbb{R}$ is defined by
\[\lambda(x)=\int_{0}^{x}e^{-\frac{1}{w^{2}(1-w)^{2}}}dw.\]
Note that $\lambda$ satisfies $\lambda^{(n)}(0)=\lambda^{(n)}(1)=0$ for all $n\geq 1$.
Since the restriction of this diffeomorphism on the boundary is the identity map, we can trivially extend this diffeomorphism to a diffeomorphism on $M$. Note that we have
\[(D_{x_{i}y})_{*} : x_{j}\mapsto x_{j} + \delta_{ij}y, \ z_{j}\mapsto z_{j}, \ y\mapsto y, \ t\mapsto t.\]

Then we have 
\begin{eqnarray*}
(D_{x_{i}y})_{*}(T_{x_{i}y}) & = & T_{x_{i}y}, \\
(D_{x_{i}y})_{*}(T_{z_{i}t}) & = & T_{z_{i}t}, \\
(D_{x_{i}y})_{*}(T_{x_{i}t}) & = & T_{x_{i}t}-(-F), \\
(D_{x_{i}y})_{*}(-T_{z_{i}y}) & = & (-T_{z_{i}y}), \\
(D_{x_{i}y})_{*}(S) & = & S+(-T_{z_{i}y}), \\
(D_{x_{i}y})_{*}(-F) & = & (-F)
\end{eqnarray*}
and hence, we have
\[(D_{x_{i}y})_{*}(\sigma)=(\dots, a_{i}, b_{i}, c_{i}, d_{i}+e, \dots ,e ,f-c_{i}),\]
where $\sigma$ is $(a_{1}, b_{1}, c_{1}, d_{1}, \dots, a_{g}, b_{g}, c_{g}, d_{g}, e, f)\in \mathbb{Z}^{4g+2}$.

We define self-diffeomorphisms $f_{y}$ and $f_{t}$ of $T^{2}\cong (\mathbb{R}/\mathbb{Z})^{2}$ by
\begin{eqnarray*}
f_{y}(y, t) & = & (y, t+y), \\
f_{t}(y, t) & = & (y+t, t)
\end{eqnarray*}
and denote diffeomorphisms $Id_{\Sigma_{g}}\times f_{y}$ and $Id_{\Sigma_{g}}\times f_{t}$ of $M$ simply by $f_{y}$ and $f_{t}$.

Then we have
\begin{eqnarray*}
(f_{y})_{*}(T_{x_{i}y}) & = & T_{x_{i}y}+T_{x_{i}t}, \\
(f_{y})_{*}(T_{z_{i}t}) & = & T_{z_{i}t }, \\
(f_{y})_{*}(T_{x_{i}t}) & = & T_{x_{i}t}, \\
(f_{y})_{*}(-T_{z_{i}y}) & = & (-T_{z_{i}y})-T_{z_{i}t}, \\
(f_{t})_{*}(T_{x_{i}y}) & = & T_{x_{i}y}, \\
(f_{t})_{*}(T_{z_{i}t}) & = & T_{z_{i}t}-(-T_{z_{i}y}), \\
(f_{t})_{*}(T_{x_{i}t}) & = & T_{x_{i}t}+T_{x_{i}y}, \\
(f_{t})_{*}(-T_{z_{i}y}) & = & (-T_{z_{i}y})
\end{eqnarray*}
for all integers $i$ with $1\leq i\leq g$. Hence we have
\begin{eqnarray*}
(f_{y})_{*}(\sigma) & = & (a_{1}, b_{1}-d_{1}, c_{1}+a_{1}, d_{1}, \dots, a_{g}, b_{g}-d_{g}, c_{g}+a_{g}, d_{g}, e, f), \\
(f_{t})_{*}(\sigma) & = & (a_{1}+c_{1}, b_{1}, c_{1}, d_{1}-b_{1}, \dots, a_{g}+c_{g}, b_{g}, c_{g}, d_{g}-b_{g}, e, f).
\end{eqnarray*}

Let $\sigma$ be $(a_{1}, b_{1}, c_{1}, d_{1}, \dots, a_{g}, b_{g}, c_{g}, d_{g}, e, f)\in \mathbb{Z}^{4g+2}$. For each integer $i$, we denote $(0, \dots, 0, a_{i}, b_{i}, c_{i}, d_{i}, 0, \dots, 0, 0, 0)\in \mathbb{Z}^{4g+2}$ by $\sigma_{i}$. Note that we have $\sigma_{i}^{2}=2(a_{i}b_{i}+c_{i}d_{i})$ and $\sigma\cdot\sigma = \sigma^{2} = \sum_{i = 1}^{g}\sigma_{i}^{2} + 2ef$. For the rest of this paper, we use $a_{1}$, $b_{1}$, \dots, $c_{g}$, $d_{g}$, $e$, $f$ as a dual basis of a basis $T_{x_{1}y}$, $T_{z_{1}t}$, \dots, $T_{x_{g}t}$, $(-T_{z_{g}y})$, $S$, $(-F)$.

\begin{lemma}\label{lemma:1}
For any homology class $\sigma$ in $H_{2}(M)$, there exists a diffeomorphism $h:M\to M$ which satisfies $b_{i}(h_{*}(\sigma))=d_{i}(h_{*}(\sigma))=0$ for all $i\geq 2$.
\end{lemma}

\begin{proof}
If $\sigma_{i}^{2}=0$ for an integer $i$, we map $\sigma$ to a class with $b_{i}=d_{i}=0$ by using $R_{x_{j}}^{\pm 1}$ and $R_{z_{j}}^{\pm 1}$ repeatedly. Therefore, it suffices to show that there exists a diffeomorphism $h:M\to M$ which satisfies $h_{*}(\sigma)_{i}^{2}=0$ for all $i\geq 2$.

Suppose that $\sigma_{i}^{2}\neq 0$ for an integer $i\geq 2$. By using $R_{x_{j}}^{\pm 1}$ and $R_{z_{j}}^{\pm 1}$ $(j=1,i)$, we map $\sigma$ to a class $\sigma '$ with $d_{1}(\sigma ')=d_{i}(\sigma ')=0$ (cf. the Euclid algorithm). Then we map $\sigma '$ to a class $\sigma ''$ with $a_{1}(\sigma '')=a_{1}(\sigma ')$, $a_{i}(\sigma '')=a_{i}(\sigma ')$ and $d_{1}(\sigma '')=d_{i}(\sigma '')=a_{1}(\sigma ')+a_{i}(\sigma ')$ by using $R_{z_{1}+z_{i}}$. Note that $\gcd (a_{1}(\sigma ''), d_{1}(\sigma ''))=\gcd (a_{i}(\sigma ''), d_{i}(\sigma ''))$.

Then we map $\sigma ''$ to a class $\sigma '''$ with $d_{1}(\sigma ''')=d_{i}(\sigma ''')=0$ and
$a_{1}(\sigma ''')=a_{i}(\sigma ''')$ by using $R_{x_{j}}^{\pm 1}$ and $R_{z_{j}}^{\pm 1}$. We have $a_{i}((R_{z_{1}+x_{i}})_{*}(\sigma '''))=d_{i}((R_{z_{1}+x_{i}})_{*}(\sigma '''))=0$ and hence, $((R_{z_{1}+x_{i}})_{*}(\sigma '''))_{i}^{2}=0$. We conclude that there exists a diffeomorphism $h$ with $h_{*}(\sigma)_{i}^{2}=0$ for all $i\geq 2$.
\end{proof}

\begin{lemma}\label{lemma:2}
For any homology class $\sigma$ in $H_{2}(M)$, there is a diffeomorphism $h:M\to M$ which satisfies $b_{i}(h_{*}(\sigma))=d_{i}(h_{*}(\sigma))=0$ for all $i\geq 2$, $a_{1}(h_{*}(\sigma))\mid b_{1}(h_{*}(\sigma))$, $a_{1}(h_{*}(\sigma))\mid e(h_{*}(\sigma))$ and $c_{1}(h_{*}(\sigma))=d_{1}(h_{*}(\sigma))=0$.
\end{lemma}

\begin{proof}
By  Lemma \ref{lemma:1}, we may suppose that $\sigma$ satisfies $b_{i}(\sigma)=d_{i}(\sigma)=0$ for all $i\geq 2$.

Furthermore we may suppose that $d_{1}(\sigma)=0$ by using $R_{x_{1}}^{\pm 1}$ or $R_{z_{1}}^{\pm 1}$ repeatedly. We map $\sigma$ to a class $\sigma'$ with $d_{1}(\sigma ')=e(\sigma ')$ by using $D_{x_{1}y}$. Then we have
\[\gcd(a_{1}(\sigma '), b_{1}(\sigma '), c_{1}(\sigma '), d_{1}(\sigma '))=\gcd(a_{1}(\sigma '), b_{1}(\sigma '), c_{1}(\sigma '), d_{1}(\sigma '), e(\sigma ')).\]

Now, we use the following procedures to map $\sigma'$ to a new class denoted by the same symbol $\sigma'$.
\begin{enumerate}[$(a)$]
\item \ Map $\sigma'$ to a class with $c_{1}(\sigma') = 0$ by using $f_{y}^{\pm 1}$ and $f_{t}^{\pm 1}$ repeatedly.
\item \ Map $\sigma'$ to a class with $d_{1}(\sigma') = 0$ by using $R_{x_{1}}^{\pm 1}$ and $R_{z_{1}}^{\pm 1}$ repeatedly.
\item \ If $c_{1}(\sigma') = d_{1}(\sigma') = 0$, map $\sigma'$ to a class with $c_{1}(\sigma') = b_{1}(\sigma')$ by using $R_{x_{1}}$.
\end{enumerate}
Using the procedure $(a)$ reduces the value $|a_{1}(\sigma')|$ except for the case $a_{1}(\sigma')\mid c_{1}(\sigma')$. The same thing holds for the procedure $(b)$. Hence, we have a class with $c_{1}(\sigma') = d_{1}(\sigma') = 0$ after finite times reduction by procedures $(a)$ and $(b)$. If $a_{1}(\sigma')\mid \hspace{-.69em}/ b_{1}(\sigma')$, after using the procedure $(c)$, we may reduce $|a_{1}(\sigma')|$ again and map to a class with $c_{1}(\sigma') = d_{1}(\sigma') = 0$. Finally we have a class $\sigma''$ with $c_{1}(\sigma'') = d_{1}(\sigma'') = 0$ and $a_{1}(\sigma'')\mid b_{1}(\sigma'')$.

Since these procedures preserve $\gcd(a_{1}(\sigma '), b_{1}(\sigma '), c_{1}(\sigma '), d_{1}(\sigma '))$ and $e(\sigma')$, we also have $a_{1}(\sigma'')\mid e(\sigma'')$.
\end{proof}

%------------------------------------------------

\subsection{Circle sum operations}\label{section:1}

In this subsection we explain the circle sum operation which makes a connected closed surface from two connected closed surfaces with positive genera. For details see Bang-He Li and Tian-Jun Li~\cite{bt1}.
Let $W$ be a $4$-manifold and let $\Sigma$ and $\Sigma'$ be closed oriented surfaces of positive genera disjointly embedded in $W$. Suppose that there is an embedded annulus $q:I\times S^{1}\to W$ with the following conditions:
\begin{itemize}
\item $q(I\times S^{1})\cap\Sigma=q(\{0\}\times S^{1})$ is homologically nontrivial in $H_{1}(\Sigma)$.
\item $q(I\times S^{1})\cap\Sigma'=q(\{1\}\times S^{1})$ is homologically nontrivial in $H_{1}(\Sigma')$.
\item There is a vector field $V$ on  $q(I\times S^{1})\subset W$ such that $V$ is not tangent to $q(I\times S^{1})$ at each point and tangent to $\Sigma$ and $\Sigma'$ on the boundary.
\end{itemize}
Take a parallel copy $q':I\times S^{1}\to W$ of the embedded annulus $q:I\times S^{1}\to W$ in accordance with the vector field $V$. We may suppose that
\begin{itemize}
\item $q'(I\times S^{1})\cap\Sigma=q'(\{0\}\times S^{1})$.
\item $q'(I\times S^{1})\cap\Sigma'=q'(\{1\}\times S^{1})$.
\item $q(I\times S^{1})\cap q'(I\times S^{1})=\emptyset$.
\end{itemize}
Remove open annuli enclosed by $q(\{0\}\times S^{1})$ and $q'(\{0\}\times S^{1})$ from $\Sigma$ and ones enclosed by $q(\{1\}\times S^{1})$ and $q'(\{1\}\times S^{1})$ from $\Sigma'$. Connect these two remaining surfaces via embedded annuli and after smoothing, we have an embedded closed oriented surface of genus $g(\Sigma)+g(\Sigma')-1$ representing $\pm([\Sigma]+[\Sigma'])$ or $\pm([\Sigma]-[\Sigma'])\in H_{2}(W)$. Note that the vector field $V$ can be identified with a section of the normal bundle $\cong I\times S^{1}\times\mathbb{C}$ and under this identification, we have another vector field $V'$ satisfying the same properties as above, given by $V'(s, t)=e^{i\pi s}V(s, t)$. Then the circle sum with respect to $V'$ gives an embedded closed oriented surface of genus $g(\Sigma)+g(\Sigma')-1$ representing $\pm([\Sigma]-[\Sigma'])$ or $\pm([\Sigma]+[\Sigma'])\in H_{2}(W)$. Hence we do not have to worry about a sign seriously.

\begin{example} \label{ex:1}
Two loops in Figure \ref{figure:2} represent embedded tori $T_{1}:T^{2}\to\Sigma_{2}\times T^{2}:(s, t)\mapsto (x_{1}(s), t, 0)$ and $T_{2}:T^{2}\to\Sigma_{2}\times T^{2}:(s, t)\mapsto (x_{2}(s), t, 0)$ and the path $\gamma:I\to \Sigma_{2}$ represents an embedded annulus $I\times S^{1}\to\Sigma_{2}\times T^{2}:(s, t)\mapsto(\gamma(s),t , 0)$.  Then we take a vector field $V$ satisfying the above condition by $V(s, t)=(v(s), t, 0)$, where we identified $\gamma(I)\subset\Sigma_{2}$ with $I$ and $v$ is a vector field on $I\subset\Sigma_{2}$ which is transverse to $I$ and tangent to $x_{1}$ and $x_{2}$ as in Figure 2. Hence we can perform the circle sum operation between $T_{1}$ and $T_{2}$. 

%fig2
\begin{figure}[h]
\centering
\includegraphics[width=8cm]{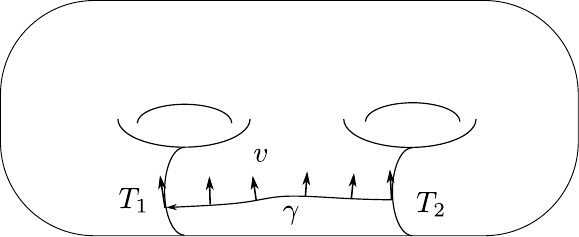}
\caption{\label{figure:2}}
\end{figure}
\end{example}

\begin{example} \label{ex:2}
We may take a circle sum operation between a torus $T_{1}$ defined as in Example \ref{ex:1} and a section $\Sigma_{g}\times \{(0, \tfrac{1}{2})\}\subset\Sigma_{g}\times T^{2}$. Take a path $\gamma : I\to T^{2} : t\mapsto (0, \tfrac{t}{2})$ and a vector field $V$ on the annulus $S^{1}\times I$ $($which is given by $x_{1}\times \gamma$$)$ by
\[S^{1}\times I\to \mathbb{R}^{4} : (x, t)\mapsto \left(0, \sin \dfrac{\pi t}{2}, \cos \dfrac{\pi t}{2}, 0\right),\]
where we have identified a neighborhood $x_{1}\times I$ of $x_{1}\subset \Sigma_{g}$ with $S^{1}\times I$. The annulus and the vector field $V$ satisfies the required conditions and we may perform the circle sum operation.
\end{example}

\begin{example} \label{ex:3}
Take real numbers $0 < q_{1} < \dots < q_{n} < 1$ and $n - 1$ distinctive points $t_{1}, \dots, t_{n - 1}\in S^{1}$. For parallel embedded tori $T^{2}\times \{q_{1}, \dots, q_{n}\}\times \{\tfrac{1}{2}\}\subset T^{2}\times I\times I$, take an annulus $S^{1}\times \{t_{i}\}\times [q_{i}, q_{i + 1}]\times \{\tfrac{1}{2}\}$ and a vector field $V_{i}$ on $S^{1}\times [q_{i}, q_{i + 1}]$ given by
\[S^{1}\times [q_{i}, q_{i + 1}]\to \mathbb{R}^{4} : (x, t)\mapsto (0, 1, 0, 0)\]
for each $1 \leq i \leq n - 1$. These annuli and vector fields satisfy required conditions and we may perform circle sum operations for parallel embedded tori.

Note that, if we take $n - 1$ distinctive points $t_{1}, \dots, t_{n - 1}\in S^{1}$ in a sufficiently small interval $I'$ of $S^{1}$, circle sum operations do not affect the original embedded tori except for the sufficiently small region $S^{1}\times I'\times I^{2}$.

\end{example}

%------------------------------------------------

\subsection{The generalized adjunction inequality}
We use the following theorems to get lower bounds for the minimal genus function.

\begin{theorem}[Kronheimer-Mrowka~\cite{km}]\label{thm:km}
Let $W$ be a closed $4$-manifold with $b_{2}^{+}(W)\geq 2$ and let $\Sigma\subset W$ be an embedded connected closed surface of genus $g(\Sigma)$ with $[\Sigma]^{2}\geq 0$ and $[\Sigma]\neq 0$. Then we have\[2g(\Sigma)-2\geq [\Sigma]^{2}+|[\Sigma]\cdot K| \] for any Seiberg-Witten basic class $K$.
\end{theorem}

\begin{theorem}[Taubes~\cite{c}]
Let $(W, \omega)$ be a closed symplectic $4$-manifold with $b_{2}^{+}(W)\geq 2$. Then the first Chern class $c_{1}(\omega)$ of the associated complex structure on $W$ has Seiberg-Witten invariant equal to $\pm 1$.
\end{theorem}

Let $\omega_{1}$ be a volume form of $\Sigma_{g}$ and let $\omega_{2}$ be a volume form of $T^{2}$. Then $\omega=P_{1}^{*}\omega_{1}-P_{2}^{*}\omega_{2}$ is a symplectic form on $M$, where $P_{1}:M\to\Sigma_{g}$ and $P_{2}:M\to T^{2}$ are projections. We have a lower bound by applying the above theorems to this symplectic structure. Note that, since the associated complex structure on $M$ is the product of the associated complex structure on each components up to isomorphism, we have \[c_{1}(\omega)=P_{1}^{*}c_{1}(\omega_{1})-P_{2}^{*}c_{1}(\omega_{2})=P_{1}^{*}c_{1}(\omega_{1}).\]

\begin{corollary}
For any class $\sigma\in H_{2}(M)\setminus\{0\}$ and any embedded connected closed surface $\Sigma\subset M$ with $[\Sigma]=\sigma$, we have \[g(\Sigma)\geq 1+\dfrac{1}{2}|\sigma\cdot\sigma|+(g-1)|F\cdot\sigma|.\]
\end{corollary}

\begin{proof}
Since $c_{1}(\omega)=P_{1}^{*}c_{1}(\omega_{1})$, we have $c_{1}(\omega)(T_{uv}) = 0$ for all $u\in H_{1}(\Sigma_{g})$ and $v \in H_{1}(T^{2})$, $c_{1}(\omega)(S)=2-2g$ and $c_{1}(\omega)(F)=0$. By the above theorems, we have\[2g(\Sigma)-2\geq |\sigma\cdot\sigma|+(2g-2)|F\cdot\sigma|\] for any classes with $\sigma^{2}\geq 0$. Since $M$ has an orientation reversing self-diffeomorphism, we may apply Theorem \ref{thm:km} to $M$ with the opposite orientation and obtain the inequality for any classes with $\sigma^{2}\leq 0$.
\end{proof}

We need the following lemma to improve this lower bound.

\begin{lemma}\label{lemma:3}
Let $g\geq 2$ and let $\phi:T^{2}\to M$ be a continuous map. Then, there are $u\in H_{1}(\Sigma_{g})$, $v\in H_{1}(T^{2})$ and $n\in \mathbb{Z}$ such that \[\sigma:=\phi_{*}[T^{2}]=u\otimes v+n(-F)\in H_{2}(M).\]
\end{lemma}

\begin{proof}
Let $P_{1}:M\to\Sigma_{g}$ and $P_{2}:M\to T^{2}$ be projections to each component. Since $g\geq 2$, we have $\mathrm{Im}(P_{1*}\circ\phi_{*}:\pi_{1}(T^{2})\to\pi_{1}(\Sigma_{g}))\cong\mathbb{Z} \mbox{ or } \{1\}$. We may assume that $P_{1*}\circ\phi_{*}[\{0\}\times S^{1}]=1\in\pi_{1}(\Sigma_{g})$, $\phi(\{0\}\times S^{1})\subset\{*\}\times T^{2}$ and $\phi(0, 0)=(*, 0, 0)$.

Let $\theta_{1}=\phi_{*}[S^{1}\times\{0\}]\in\pi_{1}(M)$ and $\theta_{2}=\phi_{*}[\{0\}\times S^{1}]\in\pi_{1}(M)$. Define a map $\psi:T^{2}\to M$ by
\[\psi(s, t)=\left\{\begin{array}{ll}
\phi(2s, t) & \left(0\leq s\leq \dfrac{1}{2}\right) \\
(*,P_{2}\circ\phi(0, t)-P_{2}\circ\phi(2s-1, 0)) & \left(\dfrac{1}{2}\leq s\leq 1\right)
\end{array}\right..\]
Then we have $\psi_{*}[S^{1}\times\{0\}]=\theta_{1}-P_{2*}\theta_{1}=P_{1*}\theta_{1}\in\pi_{1}(M)$ and $\psi_{*}[T^{2}]=\sigma-n(-F)$ for some $n\in\mathbb{Z}$.

Let $\gamma_{1}\subset\Sigma_{g}$ be a closed curve with $[\gamma_{1}]=P_{1*}\theta_{1}\in\pi_{1}(\Sigma_{g})$ and let $\gamma_{2}\subset T^{2}$ be a closed curve with $[\gamma_{2}]=P_{2*}\theta_{2}\in\pi_{1}(T^{2})$. Define a map $\bar{\psi}:T^{2}\to M$ by \[\bar{\psi}(s, t)=(\gamma_{1}(s), \gamma_{2}(t)).\] Since $\pi_{2}(M)=\{0\}$, $\psi$ and $\bar{\psi}$ are homotopic. Now, we have
\[\sigma=[P_{1*}\theta_{1}]\otimes[P_{2*}\theta_{2}]+n(-F).\]
\end{proof}

\begin{remark}
For the case $g=1$, we also have $\phi_{*}[T^{2}]=u\otimes v+n(-F)$ under the additional condition $F\cdot\sigma=0$.
\end{remark}

%------------------------------------------------

\section{the proof of the main theorem}
\subsection{Proof for the case $F\cdot\sigma\neq0$}

By Lemma \ref{lemma:2}, we may suppose that $\sigma$ satisfies $b_{i}(\sigma)=d_{i}(\sigma)=0$ for all $i\geq 2$, $c_{1}(\sigma)=d_{1}(\sigma)=0$, $a_{1}(\sigma)\mid b_{1}(\sigma)$ and $a_{1}(\sigma)\mid e(\sigma)$. By the assumption $F\cdot\sigma\neq0$, we may assume that $e(\sigma)\neq 0$ and $a_{1}(\sigma)\neq 0$.

Now we construct an embedded surface in $M$ as follows. Let $b'=-\dfrac{e(\sigma)f(\sigma)}{a_{1}(\sigma)}$ and $n=\gcd(a_{1}(\sigma), f(\sigma))$. Take an embedded torus $\Tilde{\Sigma}\subset T^{4}=\Sigma_{1}\times T^{2}$ given by
\[T^{2}\to T^{4}:(u, v)\mapsto\left(\dfrac{a_{1}(\sigma)}{n}u, \dfrac{e(\sigma)}{a_{1}(\sigma)}v, v, \dfrac{f(\sigma)}{n}u\right).\]
Then, we have $[\Tilde{\Sigma}]=\dfrac{1}{n}(a_{1}(\sigma), b', 0, 0, e(\sigma), f(\sigma))\in H_{2}(\Sigma_{1}\times T^{2})$ and $[\Tilde{\Sigma}]^{2}=0$. Take $n$ parallel copies of $\Tilde{\Sigma}$ in $T^{4}$. We have an embedded torus $\Tilde{\Sigma}'\subset T^{4}$ by taking the circle sum around $S^{1}\hookrightarrow\Tilde{\Sigma}:\theta\mapsto\left(\dfrac{a_{1}(\sigma)}{n}\theta, 0, 0, \dfrac{f(\sigma)}{n}\theta\right)$. $($See Example \ref{ex:3}. Note that a tubular neighborhood of $\Tilde{\Sigma}$ is diffeomorphic to $T^{2}\times I\times I$.$)$ We may assume that $\Tilde{\Sigma}'$ intersects with $\{p=(\tfrac{1}{2}, \tfrac{1}{2})\}\times T^{2}$ transversely in $\bar{e}=|e(\sigma)|$ points $\{p\}\times \{q_{1}\}$, $\dots$, $\{p\}\times \{q_{\bar{e}}\}$, where $q_{1}$, $\dots$, $q_{\bar{e}}$ are the second components of intersection points in $\Sigma_{1}\times T^{2}$.

Now, we glue $(\Sigma_{g-1,1}\times T^{2}, \Sigma_{g-1,1}\times\{q_{1}, \dots, q_{\bar{e}}\})$ and $(\Sigma_{1}\times T^{2}, \Tilde{\Sigma}')\setminus(U_{p}\times T^{2})$ trivially, where $\Sigma_{g-1,1}$ is a compact oriented surface of genus $g-1$ with one boundary and $U_{p}\subset\Sigma_{1}$ is a small open disk around $p$. We have an embedded surface $\bar{\Sigma}\subset M$ with $g(\bar{\Sigma})=1+(g-1)\bar{e}$ and $[\bar{\Sigma}]=(a_{1}(\sigma), b', 0, 0, \dots, e(\sigma), f(\sigma))\in H_{2}(M)$.

We identify $T_{z_{1}t}\subset\Sigma_{1,1}\times T^{2}\subset M$ with $T^{2}\hookrightarrow\Sigma_{1,1}\times T^{2}:(u, v)\mapsto(0, u, \tfrac{1}{2}, v)$. Then, by the construction of $\bar{\Sigma}$, $T_{z_{1}t}$ intersects with $\bar{\Sigma}$ transversely in $\bar{a}=|a_{1}(\sigma)|$ points. By taking $(b_{1}(\sigma)-b')$ parallel copies of $T_{z_{1}t}$ and smoothing all intersection points with $\bar{\Sigma}$, we have a connected embedded surface $\bar{\Sigma}'$. Note that we have 
\begin{eqnarray*}
g(\bar{\Sigma}') & = & 1+|b_{1}(\sigma)-b'|\bar{a}+(g-1)\bar{e} \\
& = & 1+|a_{1}(\sigma)b_{1}(\sigma)+e(\sigma)f(\sigma)|+(g-1)|F\cdot\sigma| \\
& = & 1+\dfrac{1}{2}|\sigma\cdot\sigma|+(g-1)|F\cdot\sigma|
\end{eqnarray*}
and $[\bar{\Sigma}']=(a_{1}(\sigma), b_{1}(\sigma), 0, 0, \dots, e(\sigma), f(\sigma))\in H_{2}(M)$.

For each integer $2\leq i\leq g$, let $T_{i}$ be an embedded torus in $M$ defined by $T^{2}\to M:(u, v)\mapsto \left(x_{i}(u), \dfrac{a_{i}(\sigma)}{m_{i}}v, \dfrac{c_{i}(\sigma)}{m_{i}}v\right)$, where $m_{i}=\gcd(a_{i}(\sigma), c_{i}(\sigma))$. Note that we may assume that these tori and $\bar{\Sigma}'$ are disjoint. By taking $m_{i}$ parallel copies of $T_{i}$ for each $i$ and taking the circle sum operation between these tori and $\Sigma_{g-1,1}\times\{q_{1}\}\subset\bar{\Sigma}'$ as in Example \ref{ex:2}, we have a connected embedded surface $\Sigma$. The surface $\Sigma$ satisfies $[\Sigma]=\sigma$ and $g(\Sigma)=1+\dfrac{1}{2}|\sigma\cdot\sigma|+(g-1)|F\cdot\sigma|$.

%------------------------------------------------

\subsection{Proof for the case $F\cdot\sigma=0$}

If $\sigma^{2}\neq0$, by Lemma \ref{lemma:2}, we may suppose that $b_{i}(\sigma)=d_{i}(\sigma)=0$ for all $i\geq 2$, $c_{1}(\sigma)=d_{1}(\sigma)=0$, $a_{1}(\sigma)\neq 0$ and $b_{1}(\sigma)\neq 0$. Then we represent $\sigma$ by immersed tori as in Figure \ref{figure:3a}.

%fig3a
\begin{figure}[h]
\centering
\includegraphics[width=12cm]{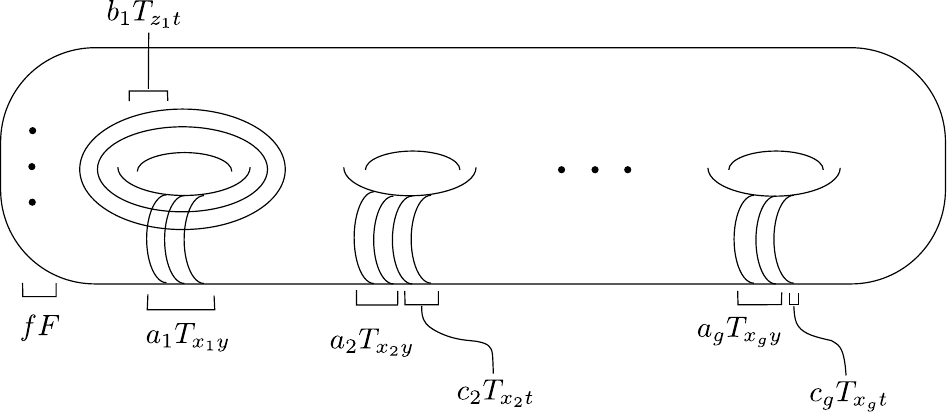}
\caption{\label{figure:3a}}
\end{figure}

Take circle sum operations in accordance with the diagram in Figure \ref{figure:3b}. The circle sum operation between a fiber $F$ and a torus $T_{z_{1}t}$ is performed as in Example \ref{ex:2}. $($Note that exchanging coordinate systems for fibers and sections gives a similar situation explained in Example \ref{ex:2}.$)$ And the others are performed as in Example \ref{ex:1} or Example \ref{ex:3}.

%fig3b
\begin{figure}[h]
\centering
\includegraphics[width=12cm]{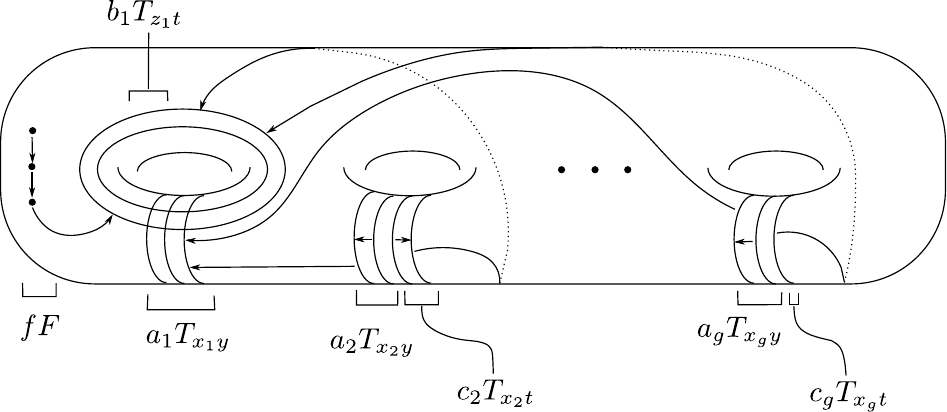}
\caption{\label{figure:3b}}
\end{figure}

After smoothing all intersections between $a_{1}(\sigma)T_{x_{1}y}$ and $b_{1}(\sigma)T_{z_{1}t}$, we have an embedded connected surface $\Sigma$ with $g(\Sigma)=1+\dfrac{1}{2}\mid\sigma\cdot\sigma\mid$.

Suppose that $\sigma\neq 0$ and $\sigma=u\otimes v+n(-F)$ for some $u\in H_{1}(\Sigma_{g})$, $v\in H_{1}(T^{2})$ and $n\in\mathbb{Z}$. If $u\otimes v=0$, clearly we can represent $\sigma$ by an embedded torus, and hence we may suppose that $u\neq 0$ and $v\neq 0$.

Let $k=\operatorname{div}(u)$, $l=\operatorname{div}(v)$, $u'=\dfrac{u}{k}$ and $v'=\dfrac{v}{l}$, where $\operatorname{div}(\cdot)$ is the divisibility. Take simple closed curves $\gamma_{1}\subset\Sigma_{g}$ with $[\gamma_{1}]=u'$ and $\gamma_{2}\subset T^{2}$ with $[\gamma_{2}]=v'$ and define an embedding $\phi:T^{2}\to M$ by $(s,t)\mapsto(\gamma_{1}(s), \gamma_{2}(t))$. Then we have $\phi_{*}[T^{2}]=k^{-1}l^{-1}u\otimes v$. Take $kl$ parallel copies of this embedded torus and $n$ parallel copies of a fiber. Then we have an embedded torus representing $\sigma$ by connecting all these tori using the circle sum.

Finally, suppose that $\sigma\cdot\sigma=0$ and $\sigma$ is not of the form $u\otimes v+n(-F)$ for all $u\in H_{1}(\Sigma_{g})$, $v\in H_{1}(T^{2})$ and $n\in\mathbb{Z}$. It suffices to show that $\sigma$ is represented by an embedded genus $2$ surface $\Sigma_{2}\subset M$. By the proof of Lemma \ref{lemma:1} and $\sigma^{2}=0$, we may suppose that $b_{i}(\sigma)=d_{i}(\sigma)=0$ for all $i\geq 1$. Now we represent $\sigma$ by embedded tori as in Figure \ref{figure:4} and we get (at most) two tori by taking the circle sum in accordance with the figure. Note that each torus represents $\sum_{i=1}^{g}a_{i}(\sigma)T_{x_{i}y}+f(\sigma)F$ and $\sum_{i=1}^{g}c_{i}(\sigma)T_{x_{i}t}$. 

%fig4
\begin{figure}[h]
\centering
\includegraphics[width=12cm]{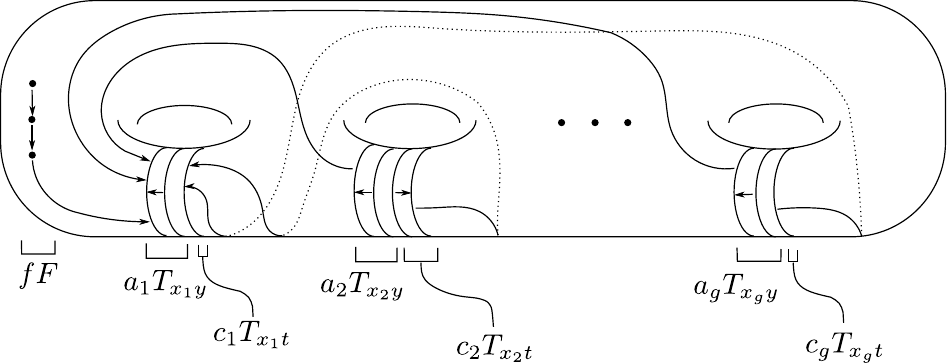}
\caption{\label{figure:4}}
\end{figure}

Then we have an embedded $\Sigma_{2}$ representing $\sigma$ by taking the connected sum of these tori.

%------------------------------------------------

\section{corollaries}\label{section:2}
\subsection{Minimal complexity functions}

\begin{definition}
Let $M$ be a closed oriented $4$-manifold. Define a map $x:H_{2}(M)\to\mathbb{Z}$ called the minimal complexity function by
\begin{eqnarray*}
\begin{array}{ll}
x(\sigma)=\min\{\chi_{-}(\Sigma)\mid\Sigma\subset M: & \mbox{a smooth embedded surface} \\
 & \mbox{$($not necessarily connected$)$ representing } \sigma\},
\end{array}
\end{eqnarray*}
where $\chi_{-}(\Sigma)$ is the complexity of a surface $\Sigma$. We also define a map $x_{c}:H_{2}(M)\to\mathbb{Z}$ by \[x_{c}(\sigma)=\min\{\chi_{-}(\Sigma)\mid\Sigma\subset M:\mbox{a smooth embedded connected surface representing } \sigma\}.\]

\end{definition}

\begin{remark}
Note that $\chi_{-}(S^{2})=0$ and $\chi_{-}(\Sigma_{h})=2h-2$ for all $h\geq 1$.
The minimal genus function $G$ distinguishes a sphere and a torus but $x_{c}$ does not. This is the only essential difference between $G$ and $x_{c}$.
\end{remark}

For minimal complexity functions, we use the following lower bound of M. Nagel~\cite{m} which generalizes results of P. B. Kronheimer~\cite{k} and S. Friedl and S. Vidussi~\cite{fv}.

\begin{theorem}[Nagel]\label{theorem:2}
Let $N$ be a graph manifold of composite type and let $p:M\to N$ be an $S^{1}$-bundle over $N$. Then we have \[x(\sigma)\geq\mid\sigma\cdot\sigma\mid+\|p_{*}\sigma\|_{T}\] for all homology classes $\sigma\in H_{2}(M)$, where $\|\cdot\|_{T}$ is the Thurston norm.
\end{theorem}

For the definition of graph manifolds of composite type and the proof of this theorem, see M. Nagel~\cite{m}. Note that $\Sigma_{g}\times S^{1}$ ($g\geq 2$) is a graph manifold of composite type.

\begin{corollary}
Let $M=\Sigma_{g}\times T^{2}=(\Sigma_{g}\times S^{1})\times S^{1}$ with $g\geq 2$ and let $p:M\to \Sigma_{g}\times S^{1}$ be the projection to the first component. Then we have
\begin{eqnarray*}
x(\sigma) & = & |\sigma\cdot\sigma|+\|p_{*}\sigma\|_{T} \ \ (\mbox{for all } \sigma\in H_{2}(M)), \\
x_{c}(\sigma) & = & \left\{\begin{array}{ll}
|\sigma\cdot\sigma|+\|p_{*}\sigma\|_{T} & (\sigma^{2}\neq 0 \mbox{ or } F\cdot\sigma\neq0 \mbox{ or } \sigma=u\otimes v+n(-F)\mbox{ for } \\ & \mbox{ some } u\in H_{1}(\Sigma_{g}), v\in H_{1}(T^{2}) \mbox{ and } n\in\mathbb{Z}) \\
2 & (\mbox{otherwise})
\end{array}\right..
\end{eqnarray*}
\end{corollary}

\begin{proof}
Since $\|p_{*}\sigma\|_{T}=2(g-1)|\sigma\cdot F|$, the equality for $x_{c}$ follows from Theorem \ref{theorem:1}. Since a homology class $\sigma\in H_{2}(M)$ with $\sigma^{2}=0$ and $F\cdot\sigma=0$ is represented by embedded tori, the equality for $x$ follows.
\end{proof}

\begin{remark} \label{rmk:1}
For disconnected case, we may also have complexity minimizing surfaces by the construction explained in S. Friedl and S. Vidussi~\cite{fv} Lemma 4.1 and proof of Lemma 4.2. $($See also M. Nagel~\cite{mn} Section 5.6.$)$ This derives that the existence of a connected Thurston norm minimizing surface for every primitive class in $H_{2}(\Sigma_{g}\times S^{1})$ is sufficient for the sharpness of the inequality in Theorem \ref{theorem:2}.

Indeed, for any primitive class $\sigma = n[\Sigma_{g}] + m(\gamma\otimes [S^{1}]) \in H_{2}(\Sigma_{g}\times S^{1})$, where $\gamma\in H_{1}(\Sigma_{g})$ is primitive, we may take a connected Thurston norm minimizing surface representing $\sigma$. Take $n$-parallel copies of $\Sigma_{g}\times \{*\}$ and $m$-parallel copies of $\gamma\times S^{1}$. Then smoothing intersection set gives the desired surface.
\end{remark}

%------------------------------------------------

\subsection{Minimal genus function for other $T^{2}$-bundles}

\begin{theorem}
Let $p:N\to \Sigma_{g}$ be a nontrivial $S^{1}$-bundle over $\Sigma_{g}$ and let $M=N\times S^{1}$. Then we have
\begin{eqnarray*}
G(\sigma)=\left\{\begin{array}{ll}
0 & (\sigma=0) \\
1+\dfrac{1}{2}|\sigma\cdot\sigma| & (\mbox{either } \sigma^{2}\neq 0, \mbox{ or } \sigma\neq 0 \mbox{ and } \sigma=u\otimes v+n(-F) \\ & \mbox{ for some } u\in H_{1}(\Sigma_{g}), v\in H_{1}(T^{2}) \mbox{ and } n\in\mathbb{Z}) \\
2 & (otherwise) \\ 
\end{array}\right..
\end{eqnarray*}
\end{theorem}
\begin{proof}
Since any circle bundles over surfaces with a boundary are trivial, a restriction of the $T^{2}$-bundle $M\to \Sigma_{g}$ on a compact oriented surface of genus $g$ with one boundary $\Sigma_{g, 1}$ is trivial. Hence we apply the argument for constructing surfaces in trivial $T^{2}$-bundles over closed oriented surfaces. Note that $M=\Sigma_{g,1}\times T^{2}\cup_{\phi}D^{2}\times T^{2}$, where $\phi:S^{1}\times T^{2}\to S^{1}\times T^{2}$ is a diffeomorphism defined by $(\theta, s, t)\mapsto (\theta, s+m\theta,t)$ for some $m\in \mathbb{Z}\setminus\{0\}$. Consider the Mayer-Vietoris exact sequence 
\begin{eqnarray*}
&&\cdots \to H_{2}(S^{1}\times T^{2}) \stackrel{\psi_{2}}{\to} H_{2}(\Sigma_{g,1}\times T^{2})\oplus H_{2}(D^{2}\times T^{2})\to H_{2}(M) \\
&&\phantom{\cdots}\to H_{1}(S^{1}\times T^{2}) \stackrel{\psi_{1}}{\to} H_{1}(\Sigma_{g,1}\times T^{2})\oplus H_{1}(D^{2}\times T^{2})\to\cdots.
\end{eqnarray*}
Since $\psi_{1}$ is injective, we have $H_{2}(M)\cong (H_{2}(\Sigma_{g,1}\times T^{2})\oplus H_{2}(D^{2}\times T^{2}))/\operatorname{Im}\psi_{2}\cong H_{2}(\Sigma_{g,1}\times T^{2})/\langle mF\rangle\cong\mathbb{Z}^{4g}\oplus\mathbb{Z}_{m}$.

By the map $\tilde{p}=p\times Id_{S^{1}}:M\to \Sigma_{g}\times S^{1}$ we see $M$ as an $S^{1}$-bundle over a graph manifold of composite type if $g\geq 2$. Then, by Theorem \ref{theorem:2}, we have the lower bound for the complexity \[x(\sigma)\geq|\sigma\cdot\sigma|+\|\tilde{p}_{*}\sigma\|_{T}\] and hence, we have the lower bound $1+\dfrac{1}{2}|\sigma\cdot\sigma|\leq\dfrac{1}{2}(x(\sigma)+2)\leq G(\sigma)$ for all $\sigma\in H_{2}(M)\setminus\{0\}$. Note that all spheres embedded in $M$ represent $0\in H_{2}(M)$. (If $g=1$, $M$ has a symplectic structure and we also have the lower bound $1+\dfrac{1}{2}|\sigma\cdot\sigma|\leq G(\sigma)$.) We can show the same statement as in Lemma \ref{lemma:3} for this $M$ and improve the lower bound. 

Note that we may take a genus minimizing surface for $\sigma\in H_{2}(\Sigma_{g}\times T^{2})$ with $\sigma\cdot F = 0$ in $\Sigma_{g, 1}\times T^{2}$ and we may regard this surface as the genus minimizing surface in $M=\Sigma_{g,1}\times T^{2}\cup_{\phi}D^{2}\times T^{2}$.
\end{proof}

%------------------------------------------------

\subsection{Automorphisms on $H_{2}(\Sigma_{g}\times T^{2})$}

In this subsection we denote $\Sigma_{g}\times T^{2}$ by $M$ and assume that $g\geq 2$. We observe automorphisms of $H_{2}(M)$ induced by orientation preserving self-diffeomorphisms. Let $\mathcal{H}$ be the subgroup of the automorphism group ${\rm Aut}(H_{2}(M))$ defined by \[\mathcal{H}=\{\phi\in{\rm Aut}(H_{2}(M))\mid\phi^{*}Q=Q \mbox{ and }\phi^{*}G=G\},\]where $Q$ is the intersection form of $M$. Let $\theta:{\rm Diff}^{+}(M)\to\mathcal{H}$ be the obvious homomorphism from the group of orientation preserving diffeomorphisms of $M$. Then we have following theorems.

\begin{theorem}\label{theorem:431}
$\mathcal{H}/{\rm Im}\theta\cong\mathbb{Z}/2\mathbb{Z}$.
\end{theorem}

\begin{theorem}\label{theorem:432}
${\rm Im}\theta$ is generated by $(R_{\alpha})_{*}$ $(\mbox{for }\alpha\in H_{1}(\Sigma_{g}) \mbox{ primitive})$, $(D_{x_{1}y})_{*}$, $(f_{y})_{*}$, $(f_{t})_{*}$ and an automorphism $h_{*}$ induced by a diffeomorphism $h$ which is the product of orientation reversing diffeomorphisms $\Sigma_{g}\to\Sigma_{g}$ and $T^{2}\to T^{2}$.
\end{theorem}
 
To prove these theorems, we need some lemmas.

\begin{lemma}\label{lemma:4}
For any automorphism $\phi\in\mathcal{H}$, we have $\phi(F)=\pm F$.
\end{lemma}

\begin{proof}
By Lemma \ref{lemma:3}, $\phi(F)=u\otimes v+nF$ for some $u\in H_{1}(\Sigma_{g})$, $v\in H_{1}(T^{2})$ and $n\in\mathbb{Z}$. For the sake of contradiction, suppose that $u\otimes v\neq 0$. Define subsets $K$ and $K_{\phi}$ of $H_{2}(M)$ by
\begin{eqnarray*}
K&=&\{\sigma\in H_{2}(M)\mid G(\sigma)\leq 1 \mbox{ and } G(\sigma+F)\leq1\}, \\
K_{\phi}&=&\{\sigma\in H_{2}(M)\mid G(\sigma)\leq1 \mbox{ and } G(\sigma+\phi(F))\leq1\}.
\end{eqnarray*}
 
Clearly, we have $K_{\phi}=\phi(K)$. By Theorem \ref{theorem:1}, we have
\[G(\sigma)\leq 1 \Leftrightarrow G(\sigma + nF)\leq 1\] 
for any class $\sigma\in H_{2}(M)$ and integer $n\in\mathbb{Z}$. Hence we have
\begin{eqnarray*}
K&=&\{\sigma\in H_{2}(M)\mid G(\sigma)\leq1\}, \\
K_{\phi}&=&\{\sigma\in H_{2}(M)\mid G(\sigma)\leq1 \mbox{ and } G(\sigma+u\otimes v)\leq1\}.
\end{eqnarray*}

That is, $K_{\phi}\subsetneq K$ which contradicts $K_{\phi}=\phi(K)=K$.
\end{proof}

\begin{lemma}\label{lemma:5}
For any automorphism $\phi$ in $\mathcal{H}$, there exists an orientation preserving diffeomorphism $h:M\to M$ such that $h_{*}\circ\phi(S)=S$.
\end{lemma}

\begin{proof}
Since $F\cdot\phi(S)=\pm 1$ by Lemma \ref{lemma:4}, we have $e(\phi(S))=\pm 1$. We may suppose that $e(\phi(S))=1$.

As the definition of the diffeomorphism $D_{x_{i}y}$, we can define diffeomorphisms $D_{x_{i}t}$, $D_{z_{i}y}$ and $D_{z_{i}t}$ for each $1\leq i\leq g$ with the following properties: 
\begin{eqnarray*}
(D_{z_{i}t})_{*}(\sigma)&=&(\dots, a_{i}, b_{i}, c_{i}+e, d_{i}, \dots ,e ,f-d_{i}), \\
(D_{x_{i}t})_{*}(\sigma)&=&(\dots, a_{i}, b_{i}+e, c_{i}, d_{i}, \dots ,e ,f-a_{i}), \\
(D_{z_{i}y})_{*}(\sigma)&=&(\dots, a_{i}+e, b_{i}, c_{i}, d_{i}, \dots ,e ,f-b_{i}),
\end{eqnarray*}
where $\sigma=(a_{1}, b_{1}, \dots, c_{g}, d_{g}, e, f)$.

Now, define $h$ by
\[(D_{z_{1}y})^{-a_{1}(\phi(S))}\circ(D_{x_{1}t})^{-b_{1}(\phi(S))}\circ\dots\circ(D_{z_{g}t})^{-c_{g}(\phi(S))}\circ(D_{x_{g}y})^{-d_{g}(\phi(S))}.\]
We have $h_{*}\circ\phi(S)=S+nF$ for some integer $n\in\mathbb{Z}$. Since $(h_{*}\circ\phi(S))\cdot (h_{*}\circ\phi(S))=0$, we have $n=0$.
\end{proof}

\begin{remark}
Diffeomorphisms defined in the proof are realized by the composition of diffeomorphisms defined in Section $\ref{section:0}$.
\end{remark}

\begin{lemma}\label{lemma:6}
Let $h:M\to M$ be a diffeomorphism. If $h_{*}(\sigma)=\varepsilon(\sigma)\sigma$ for some $\varepsilon(\sigma)\in\{\pm 1\}$ for all $\sigma\in\{T_{x_{1}y}, T_{z_{1}t}, T_{x_{1}t}, (-T_{z_{1}y}), \dots ,T_{x_{g}y}, T_{z_{g}t}, T_{x_{g}t}, (-T_{z_{g}y})\}$, $h_{*}(S)=S$ and $h_{*}(-F)=-F$, then we have $\varepsilon(T_{x_{i}y})=\varepsilon(T_{z_{i}t})=\varepsilon(T_{x_{i}t})=\varepsilon(-T_{z_{i}y})$ for all $1\leq i\leq g$. That is, if the diffeomorphism $h$ satisfies 
\[h_{*}(a_{1}, b_{1}, \dots, c_{g}, d_{g}, e, f) = (\varepsilon_{a_{1}}a_{1}, \varepsilon_{b_{1}}b_{1}, \dots, \varepsilon_{c_{g}}c_{g}, \varepsilon_{d_{g}}d_{g}, e, f)\]
for some $\varepsilon_{a_{1}}, \dots, \varepsilon_{d_{g}}\in\{\pm 1\}$, we have $\varepsilon_{a_{i}} = \varepsilon_{b_{i}} = \varepsilon_{c_{i}} = \varepsilon_{d_{i}}$ for all $1\leq i\leq g$.
\end{lemma}

\begin{proof}
Since $h_{*}(F) = h_{*}(y)\otimes h_{*}(t) = F$, We have $h_{*}(y), h_{*}(t)\in H_{1}(T^{2})\subset H_{1}(M)$. Now, we have $h_{*}(u)=\varepsilon(u)u$ for some $\varepsilon(u)\in\{\pm 1\}$ for all $u\in\{x_{1}, z_{1}, \dots, x_{g}, z_{g}, y, t\}$ by a computation using $h_{*}(T_{uv}) = h_{*}(u)\otimes h_{*}(v)$. Note that, we have $\varepsilon(T_{uv}) = \varepsilon(u)\varepsilon(v)$. Since $h_{*}(S)=S$ and $h_{*}(-F)=-F$, we have $\varepsilon(x_{i})=\varepsilon(z_{i})$ for all $1\leq i\leq g$ and $\varepsilon(y)=\varepsilon(t)$. Hence, we have $\varepsilon(T_{x_{i}y})=\varepsilon(T_{z_{i}t})=\varepsilon(T_{x_{i}t})=\varepsilon(-T_{z_{i}y})$ for all $1\leq i\leq g$.
\end{proof}

\begin{remark}\label{remark:1}\text{}
Conversely, for any $\varepsilon_{i}\in\{\pm 1\}$ $(1\leq i\leq g)$, the automorphism $\phi$ defined by
\[\phi(a_{1}, b_{1}, \dots, c_{g}, d_{g}, e, f)=(\varepsilon_{1}a_{1}, \varepsilon_{1}b_{1}, \dots, \varepsilon_{g}c_{g}, \varepsilon_{g}d_{g}, e, f) \]
is realized by a diffeomorphism of $M$. Since the symplectic representation
\[{\rm Diff}^{+}(\Sigma_{g})\to {\rm Sp}(H_{1}(\Sigma_{g}))\cong {\rm Sp}(2g, \mathbb{Z})\]
is surjective $($See B. Farb and D. Margalit~\cite{fm} Theorem 6.4$)$, we have a diffeomorphism $h_{1}$ of $\Sigma_{g}$ with the property $h_{1*}(x_{i})=\varepsilon_{i}x_{i}$ and $h_{1*}(z_{i})=\varepsilon_{i}z_{i}$ for all $1\leq i\leq g$. Then the diffeomorphism $h=h_{1}\times Id_{T^{2}}$ satisfies $h_{*}=\phi$.

However, Lemma \ref{lemma:6} implies that the automorphism $\phi\in\mathcal{H}$ defined by
\[\phi(a_{1}, b_{1}, \dots, c_{g}, d_{g}, e, f)=(-a_{1}, -b_{1}, c_{1}, d_{1}, \dots, -a_{g}, -b_{g}, c_{g}, d_{g}, e, f)\]
$($we show that this automorphism is indeed an element of $\mathcal{H}$ in the proof of Theorem \ref{theorem:431}$)$ is not realized by a diffeomorphism. Note that the condition $h\in{\rm Diff}(M)$ is used to consider the induced map on the first homology group. This gives the difference between $\mathcal{H}$ and $\operatorname{Im}\theta$.
\end{remark}

\begin{lemma}\label{lemma:7}
For any $\phi\in\mathcal{H}$, there exists an orientation preserving diffeomorphism $h:M\to M$ such that $h_{*}\circ\phi(\sigma)=\sigma$ up to sign for all
\[\sigma\in\{T_{x_{1}y}, T_{z_{1}t}, T_{x_{1}t}, (-T_{z_{1}y}), \dots,T_{x_{g}y}, T_{z_{g}t}, T_{x_{g}t}, (-T_{z_{g}y})\},\] $h_{*}\circ\phi(S)=S$ and $h_{*}\circ\phi(-F)=-F$. That is $h_{*}\circ\phi$ satisfies
\[h_{*}\circ\phi(a_{1}, b_{1}, \dots, c_{g}, d_{g}, e, f) = (\varepsilon_{a_{1}}a_{1}, \varepsilon_{b_{1}}b_{1}, \varepsilon_{c_{1}}c_{1}, \varepsilon_{d_{1}}d_{1}, \dots, e, f)\]
for some $\varepsilon_{a_{1}}, \dots, \varepsilon_{d_{g}}\in\{\pm 1\}$.
\end{lemma}

\begin{proof}
Suppose that $\phi\in\mathcal{H}$. By Lemma \ref{lemma:5}, we may assume that $\phi(S)=S$. By the assumption and Lemma \ref{lemma:3}, for any $u\in\{x_{1},z_{1}, \dots, x_{g}, z_{g}\}$ and $v\in\{y, t\}$, there exist $\tilde{u}(u, v)\in H_{1}(\Sigma_{g})$, $\tilde{v}(u, v)\in H_{1}(T^{2})$ and $n(u, v)\in\mathbb{Z}$ such that
\[\phi(T_{uv})=\tilde{u}(u, v)\otimes\tilde{v}(u, v)+n(u, v)(-F).\]
Since $\phi(T_{uv}\pm n(u, v)(-F))=\tilde{u}(u, v)\otimes\tilde{v}(u, v)$ for a suitable sign, $\tilde{u}(u, v)$ and $\tilde{v}(u, v)$ are primitive classes.

We show that $\tilde{v}(u_{1},v)=\tilde{v}(u_{2},v)$ up to sign for any $u_{1}, u_{2}\in\{x_{1},z_{1}, \dots, x_{g}, z_{g}\}$ and $v\in\{y, t\}$. Let $u_{1}, u_{2}, u_{3}$ be elements in $\{x_{1},z_{1}, \dots, x_{g}, z_{g}\}$ with $u_{i}\neq u_{j}$ for all $1\leq i<j\leq 3$. Suppose that $\tilde{v}(u_{i}, v)\neq\pm\tilde{v}(u_{j}, v)$ for all $1\leq i<j\leq 3$. Since $G(\tilde{u}(u_{i}, v)\otimes\tilde{v}(u_{i}, v)+\tilde{u}(u_{j}, v)\otimes\tilde{v}(u_{j}, v))=1$ for all $1\leq i<j\leq 3$, we have $\tilde{u}(u_{i}, v)=\tilde{u}(u_{j}, v)$ up to sign for all $1\leq i<j\leq 3$. This contradicts the independence of $T_{u_{1}v}$, $T_{u_{2}v}$, $T_{u_{3}v}$ and $(-F)$. Suppose that $\tilde{v}(u_{1}, v)=\tilde{v}(u_{2}, v)$ up to sign and $\tilde{v}(u_{2}, v)\neq\pm\tilde{v}(u_{3}, v)$. We have $\tilde{u}(u_{1}, v)\neq\pm\tilde{u}(u_{2}, v)$ and $\tilde{u}(u_{2}, v)=\tilde{u}(u_{3}, v)$ up to sign. Hence we have $\tilde{u}(u_{1}, v)\neq\pm\tilde{u}(u_{3}, v)$ and $\tilde{v}(u_{1}, v)\neq\pm\tilde{v}(u_{3}, v)$ and
this contradicts $G(\tilde{u}(u_{1}, v)\otimes\tilde{v}(u_{1}, v)+\tilde{u}(u_{3}, v)\otimes\tilde{v}(u_{3}, v))=1$. Now, we have $\tilde{v}(u_{1},v)=\tilde{v}(u_{2},v)$ up to sign for any $u_{1}, u_{2}\in\{x_{1},z_{1}, \dots, x_{g}, z_{g}\}$ and $v\in\{y, t\}$.

Next, we show that $\tilde{u}(u, y)=\tilde{u}(u, t)$ up to sign for any $u\in\{x_{1},z_{1}, \dots, x_{g}, z_{g}\}$. Suppose that $\tilde{u}(u, y)\neq\pm\tilde{u}(u, t)$. We have $\tilde{v}(u, y)=\tilde{v}(u, t)$ and this contradicts the independence of $T_{x_{1}y}$, $T_{z_{1}y}$, \dots, $T_{x_{g}y}$, $(-T_{z_{g}y})$, $T_{ut}$ and $(-F)$. Hence, $\tilde{u}(u, y)=\tilde{u}(u, t)$ up to sign for any $u\in\{x_{1},z_{1}, \dots, x_{g}, z_{g}\}$.

Now we may suppose that $\tilde{v}(x_{1}, v)=\tilde{v}(z_{1}, v)=\dots=\tilde{v}(x_{g}, v)=\tilde{v}(z_{g}, v)$ for each $v\in\{y, t\}$ and $\tilde{v}(x_{1}, y)\cdot\tilde{v}(x_{1}, t)=1$. Since $\phi$ preserves the intersection form $Q$, $\tilde{u}(x_{1}, y)$, $\tilde{u}(z_{1}, t)$, \dots, $\tilde{u}(x_{g}, y)$ and $\tilde{u}(z_{g}, t)$ form a symplectic basis of $H_{1}(\Sigma_{g})$. Take a diffeomorphism $h_{1}:\Sigma_{g}\to\Sigma_{g}$ such that $h_{1*}(\tilde{u}(u, y))=u$ for all $\{x_{1}, x_{2}, \dots, x_{g}\}$ and $h_{1*}(\tilde{u}(u, t))=u$ for all $\{z_{1}, z_{2}, \dots, z_{g}\}$. Furthermore, take a diffeomorphism $h_{2}:T^{2}\to T^{2}$ such that $\tilde{v}(x_{1}, y)=y$ and $\tilde{v}(x_{1}, t)=t$ and define a diffeomorphism $h$ of $M$ by $h=h_{1}\times h_{2}$. The composition $h_{*}\circ\phi$ satisfies $h_{*}\circ\phi(T_{uv})=\pm T_{uv}$ up to the fiber component for all $u\in\{x_{1},z_{1}, \dots, x_{g}, z_{g}\}$ and $v\in\{y, t\}$, $h_{*}\circ\phi(S)=S$ and $h_{*}\circ\phi(-F)=(-F)$. Since $(h_{*}\circ\phi(T_{uv}))\cdot (h_{*}\circ\phi(S))=(h_{*}\circ\phi(T_{uv}))\cdot S=0$, we have $h_{*}\circ\phi(T_{uv})=\pm T_{uv}$ for all $u\in\{x_{1},z_{1}, \dots, x_{g}, z_{g}\}$ and $v\in\{y, t\}$.
\end{proof}

Now, we prove Theorem \ref{theorem:431} and Theorem \ref{theorem:432}.

\begin{proof}[Proof of Theorem \ref{theorem:431}]
Let $\phi$ be an automorphism in $\mathcal{H}$. By Lemma \ref{lemma:7}, we may suppose
\[\phi(a_{1}, b_{1}, \dots, c_{g}, d_{g}, e, f) = (\varepsilon_{a_{1}}a_{1}, \varepsilon_{b_{1}}b_{1}, \varepsilon_{c_{1}}c_{1}, \varepsilon_{d_{1}}d_{1}, \dots, e, f)\]
for some $\varepsilon_{a_{1}}, \dots, \varepsilon_{d_{g}}\in\{\pm 1\}$. Since $\phi$ preserves the intersection form, we have  $\varepsilon_{a_{i}} = \varepsilon_{b_{i}}$ and $\varepsilon_{c_{i}} = \varepsilon_{d_{i}}$ for all $1\leq i\leq g$.

We show that $\varepsilon_{a_{i}}\varepsilon_{c_{i}}=\varepsilon_{a_{j}}\varepsilon_{c_{j}}$ for all $1\leq i<j\leq g$. For the sake of contradiction, suppose that $\varepsilon_{a_{i}}\varepsilon_{c_{i}}\neq\varepsilon_{a_{j}}\varepsilon_{c_{j}}$ for some $i$ and $j$. Let $\sigma=(x_{i}+x_{j})\otimes(y+t)$. By the assumption, $\phi(\sigma)$ is equal to
\[\pm((x_{i}-x_{j})\otimes y+(x_{i}+x_{j})\otimes t) \mbox{ or } \pm((x_{i}+x_{j})\otimes y+(x_{i}-x_{j})\otimes t).\]
By Theorem \ref{theorem:1}, we have $G(\sigma)=1$ and $G(\phi(\sigma))=2$, this is a contradiction.

%--------
Now, any element of $\mathcal{H}/\operatorname{Im}\theta$ is represented by an automorphism $\phi$ of the form
\[\phi(a_{1}, b_{1}, \dots, c_{g}, d_{g}, e, f) = (\varepsilon_{a_{1}}a_{1}, \varepsilon_{b_{1}}b_{1}, \dots, \varepsilon_{c_{g}}c_{g}, \varepsilon_{d_{g}}d_{g}, e, f)\]
with $\varepsilon_{*}\in\{\pm 1\}$, $\varepsilon_{a_{i}} = \varepsilon_{b_{i}}$, $\varepsilon_{c_{i}} = \varepsilon_{d_{i}}$ and $\varepsilon_{a_{1}}\varepsilon_{c_{1}} = \varepsilon_{a_{2}}\varepsilon_{c_{2}} = \dots = \varepsilon_{a_{g}}\varepsilon_{c_{g}}$. By Remark \ref{remark:1}, we may suppose that $\varepsilon_{c_{1}} = \dots = \varepsilon_{c_{g}} = 1$ and $\varepsilon_{a_{1}} = \dots = \varepsilon_{a_{g}}$. Hence we may take a representative of any element in $\mathcal{H}/\operatorname{Im}\theta$ from
\[\phi_{+}(a_{1}, b_{1}, \dots, c_{g}, d_{g}, e, f)=(a_{1}, b_{1}, c_{1}, d_{1}, \dots, a_{g}, b_{g}, c_{g}, d_{g}, e, f) \text{ or}\]
\[\phi_{-}(a_{1}, b_{1}, \dots, c_{g}, d_{g}, e, f)=(-a_{1}, -b_{1}, c_{1}, d_{1}, \dots, -a_{g}, -b_{g}, c_{g}, d_{g}, e, f).\]
Clearly we have $[\phi_{+}] = \operatorname{Im}\theta$ and we have $[\phi_{+}] \neq [\phi_{-}]\in \operatorname{Aut}(H_{2}(M))/\operatorname{Im}\theta$ by Remark \ref{remark:1}.
%--------

Now, it suffices to show the automorphism $\phi$ defined by
\[\phi(\sigma)=(-a_{1}, -b_{1}, c_{1}, d_{1}, \dots, -a_{i}, -b_{i}, c_{i}, d_{i}, \dots, -a_{g}, -b_{g}, c_{g}, d_{g}, e, f),\]
where $\sigma=(a_{1}, b_{1}, \dots, c_{g}, d_{g}, e, f)$, is indeed an element of $\mathcal{H}$. It is obvious that this automorphism preserves the intersection form. So we have only to check that $\phi$ preserves the minimal genus function. It suffices to show that $\phi(K)=K$, where $K$ is the subset of $H_{2}(M)$ defined by $K=\{\sigma\in H_{2}(M)\mid G(\sigma)\leq1\}$. Suppose that $\sigma\in H_{2}(M)$ satisfies $G(\sigma)\leq 1$. There exist classes $u\in H_{1}(\Sigma_{g})$ and $v\in H_{1}(T^{2})$ and an integer $n\in\mathbb{Z}$ such that $\sigma=u\otimes v+n(-F)$. Let $u=\sum_{i=1}^{g}(\alpha_{i}x_{i}+\beta_{i}z_{i})$ and $v=py+qt$, where $\alpha_{i}$, $\beta_{i}$, $p$ and $q$ are integers. We have $\phi(\sigma)=(\Sigma_{i=1}^{g}(\alpha_{i}x_{i}-\beta_{i}z_{i}))\otimes(-py+qt)+n(-F)$ and, hence, $G(\phi(\sigma))\leq 1$. Now, we have $\phi(K)\subset K$. Since $\phi^{2}$ is the identity map, we have $K\subset\phi(K)$. Therefore, $\phi$ is an element of $\mathcal{H}$.
\end{proof}

\begin{proof}[Proof of Theorem \ref{theorem:432}]
All diffeomorphisms used in this subsection are given by compositions of $R_{\alpha}$ $(\mbox{for } \alpha\in H_{1}(\Sigma_{g}) \mbox{ primitive})$, $D_{x_{1}y}$, $f_{y}$, $f_{t}$ and $h$. Therefore, ${\rm Im}\theta$ is generated by  $(R_{\alpha})_{*}$, $(D_{x_{1}y})_{*}$, $(f_{y})_{*}$, $(f_{t})_{*}$ and $h_{*}$.
\end{proof}

%------------------------------------------------

\section{topologically locally-flat embeddings}\label{section:3}

In this section, we observe topologically locally-flatly embedded surfaces in $M=\Sigma_{g}\times T^{2}$. If a second homology class $\sigma\in H_{2}(M)$ has self-intersection zero, the genus minimizing smoothly embedded surface we constructed before also gives minimal genus among the topologically locally-flat surfaces. Note that, for the case $F\cdot\sigma\neq 0$, the minimality follows from the fact that any continuous map $\phi:\Sigma_{h}\to\Sigma_{g}$ with $\phi_{*}[\Sigma_{h}]=n[\Sigma_{g}]\in H_{2}(\Sigma_{g})$ for a positive integer $n$ gives $h\geq n(g-1)+1$.
 
However, we prove the following theorem.
\begin{theorem}\label{tlfe}
There are topologically locally-flat surfaces in $M$ whose genera are strictly smaller than the lower bound for smooth surfaces stated in Theorem \ref{theorem:1}.
\end{theorem}

To construct such surfaces, we need the following result shown by Lee Rudolph~\cite{l}.

\begin{theorem}[Rudolph]
For any integer $n\geq 6$, there is a connected topologically locally-flat surface $\Sigma$ in $\mathbb{CP}^{2}$ representing $n[\mathbb{CP}^{1}]\in H_{2}(\mathbb{CP}^{2})$ with genus strictly smaller than $\dfrac{1}{2}(n-1)(n-2)$ that transversely intersects with a complex line $C$ in $n$ points.
\end{theorem}

Note that if we take a sufficiently large $4$-ball $B\subset\mathbb{C}^{2}\cong\mathbb{CP}^{2}\setminus C$ $($that is, $B$ is a complement of a small tubular neighborhood of $C$ in $\mathbb{CP}^{2}$$)$, we have a connected topologically locally-flat surface $\Sigma'=\Sigma\cap B$ in $B$ so that the boundary $\partial\Sigma'\subset\partial B\cong S^{3}$ is an $(n, n)$-torus link and the genus of $\Sigma'$ is strictly smaller than $\dfrac{1}{2}(n-1)(n-2)$.

In the following, we construct topologically locally-flat surfaces whose genera are strictly smaller than the lower bound for smooth surfaces stated in Theorem \ref{theorem:1} for classes $\sigma$ of the form $\sigma=eS+fF$ with $e,f\geq 5$. For simplicity, we suppose that $g=1$, $n=6$ and $\Sigma' = \Sigma_{9, 6}$, where $\Sigma_{g, h}$ is a compact oriented surface of genus $g$ with $h$ boundaries.

\begin{proof}[Proof of Theorem\ref{tlfe}]
Suppose that $\sigma$ is represented by $e$ parallel copies of sections and $f$ parallel copies of fibers of the trivial bundle $\Sigma_{1}\times T^{2}\to \Sigma_{1}$, that is $\sigma$ is represented by
\[\Sigma_{\sigma} = (T^{2}\times\{p_{1},\dots, p_{e}\}\times \{0\})\cup (\{q_{1},\dots, q_{f}\}\times \{0\}\times T^{2}).\]
Then the intersection with $S^{1}\times \{0\}\times S^{1}\times \{0\}$ is visualized as Figure \ref{figure:7}. Take a $4$-cube $I^{4}\subset T^{4}$ which contains all intersection points.
The boundary link is
\[(S^{1}\times \{p_{1},\dots, p_{e}\}\times \{0\})\cup (\{q_{1},\dots, q_{f}\}\times \{0\}\times S^{1})\subset (S^{1}\times I^{2})\cup (I^{2}\times S^{1})\cong S^{3}.\]

%fig7
\begin{figure}[h]
\centering
\includegraphics[height=4cm]{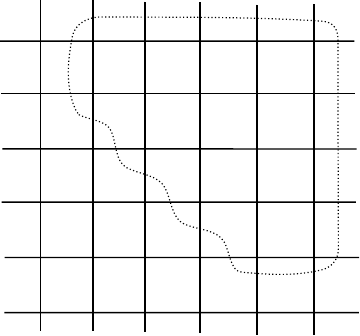}
\caption{\label{figure:7}}
\end{figure}

Attach $2$-dimensional $1$-handles to the link as in Figure \ref{figure:9} and we have a link as the right side of Figure \ref{figure:9}. Note that, this link appears as the boundary link of the suitable $4$-ball $B'$ which contains intersection points enclosed by the dotted line in Figure \ref{figure:7}. 

%fig8
\begin{figure}[h]
\centering
\includegraphics[height=5.2cm]{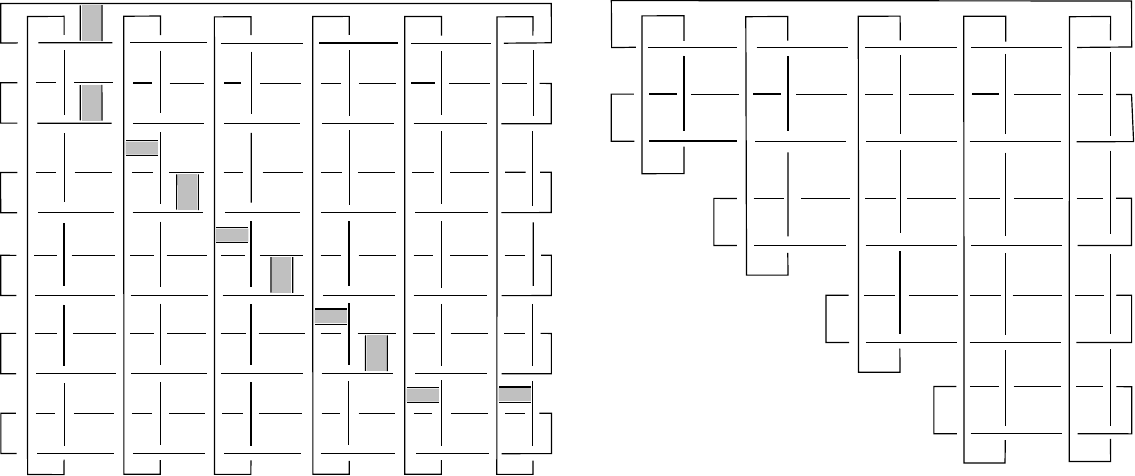}
\caption{\label{figure:9}}
\end{figure}

Remove the open ball ${\rm Int}B'$ from $(M, \Sigma_{\sigma})$ and attach $2$-dimensional $1$-handles to the remaining immersed surface as in Figure \ref{figure:6} $($We may assume that these $1$-handles are arranged in the boundary sphere$)$. Then we have a new surface with boundary $L'\subset S^{3}$ which is isotopic to the $(6,6)$-torus link. Note that this surface consists of $e + f - 10$ tori, two copies of $\Sigma_{1,1}$ and four of $\Sigma_{2, 1}$.

%fig6
\begin{figure}[h]
\centering
\includegraphics[height=5.2cm]{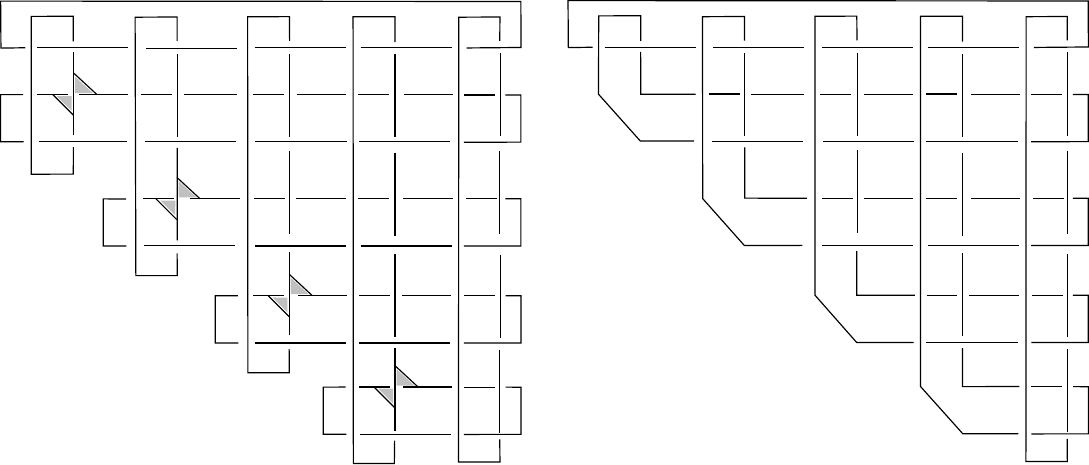}
\caption{\label{figure:6}}
\end{figure}

We have a singular surface representing $\sigma$ by gluing $M\setminus{\rm Int}B'$ and $(B, \Sigma')$ via a homeomorphism $(S^{3}, L')\to(S^{3}, \partial\Sigma')$. This surface consists of $e + f - 10$ tori and a closed oriented surface of genus $19$ and has $ef - 19$ intersection points.

Now we have a connected topologically locally-flat surface $\Sigma$ representing $\sigma$ after smoothing all singularities. Note that $e + f - 10$ singular points are used to connect each component and $ef - e - f - 9$ singular points affect the genus of the obtained surface. Hence the surface $\Sigma$ has genus 
\[(e + f - 10) + 19 + (ef - e - f - 9) = ef,\]
which is strictly smaller than $1+ef=1+\dfrac{1}{2}|\sigma\cdot\sigma|$.
\end{proof}

%------------------------------------------------

\section*{Acknowledgements}
I would like to thank Jae Choon Cha and Kouichi Yasui for helpful comments and encouragement. I would like to thank anonymous referee for valuable suggestions on the draft, especially Remark \ref{rmk:1}. I am grateful to my supervisor Takuya Sakasai for support and advice for my writing.

%------------------------------------------------


\begin{thebibliography}{99}
\bibitem{fm}
B. Farb and D. Margalit, A Primer on Mapping Class Groups, Princeton Mathematical Series Vol. 49,  Princeton University Press (2011).
\bibitem{fv}
S. Friedl and S. Vidussi, Minimal genus in $4$-manifolds with a free circle action, Adv.
Math., 250 (2014) 570--587.
\bibitem{k}
P. B. Kronheimer, Minimal genus in $S^{1}\times M^{3}$, Invent. Math., 135(1) (1999) 45--61. 
\bibitem{km}
P. B. Kronheimer and T. S. Mrowka, The genus of embedded surfaces in the projective plane, Mathematical Research Letters 1 (1994), 797--808.
\bibitem{t}
Terry Lawson, The minimal genus problem, Exposition Mathematics 15 (1997) 385--431.
\bibitem{bt1}
Bang-He Li and Tian-Jun Li, Circle-sum and minimal genus surfaces in ruled $4$-manifolds, Proceedings of the American mathematical society 135 (2007) 3745--3753.
\bibitem{bt2}
Bang-He Li and Tian-Jun Li, Minimal genus embeddings in $S^{2}$-bundles over surfaces, Mathematical Research Letters 4 (1997) 379--394.
\bibitem{m}
Matthias Nagel, Minimal genus in circle bundles over $3$-manifolds, Journal of Topology 9.3 (2016) 747--773.
\bibitem{mn}
Matthias Nagel, Surfaces of minimal complexity in low-dimensional topology, PhD thesis, University of Regensburg (2015).
\bibitem{l}
Lee Rudolph, Some topologically locally-flat surfaces in the complex projective plane, Commentarii Mathematici Helvetici 59 (1984) 592--599.
\bibitem{c}
C. H. Taubes, The Seiberg-Witten invariants and symplectic forms, Mathematical Research Letters 1 (1994), 809--822.
\end{thebibliography}
\end{document}